\documentclass[11pt]{amsart}
\usepackage{amsmath, amsfonts, amssymb, amstext, amscd, amsthm, stmaryrd}

\numberwithin{equation}{section}
\allowdisplaybreaks

\newtheorem{main}{Theorem}
\newtheorem{appell}[main]{Proposition}
\newtheorem{product}[main]{Theorem}
\newtheorem{standard}[main]{Corollary}
\newtheorem{kenter}[main]{Corollary}
\newtheorem{euler}[main]{Corollary}

\begin{document}

\title[Riordan Matrix Representations of $\gamma$ and $e$]{Riordan Matrix Representations of \\ Euler's Constant $\gamma$ and Euler's Number $e$}

\author{Edray Herber Goins}
\address{Purdue University, Department of Mathematics, Mathematical Sciences Building, 150 North University Street, West Lafayette, IN 47907-2067}
\email{egoins@math.purdue.edu}

\author{Asamoah Nkwanta}
\address{Morgan State University, Department of Mathematics, 1700 East Cold Spring Lane, Baltimore, MD 21251}
\email{asamoah.nkwanta@morgan.edu}

\begin{abstract} We show that the Euler-Mascheroni constant $\gamma$ and Euler's number $e$ can both be represented as a product of a Riordan matrix and certain row and column vectors. \end{abstract}

\dedicatory{This paper is dedicated to David Harold Blackwell (April 24, 1919 -- July 8, 2010).}

\maketitle

\section{Introduction}

It was shown by Kenter \cite{Kenter:1999p20766} that the Euler-Mascheroni constant
\begin{equation} \gamma = \displaystyle \lim_{n \to \infty} \left[ \left( \sum_{m=1}^n \dfrac 1m \right) - \ln n \right] = 0.5772156649 \dots \end{equation}

\noindent can be represented as a product of an infinite-dimensional row vector, the inverse of a lower triangular matrix, and an infinite-dimensional  column vector:
\begin{equation} \begin{pmatrix} 1 & \frac 12 & \frac 13 & \cdots & \frac 1n & \cdots \end{pmatrix} \begin{pmatrix} 1 & & & & & \\[8pt] \frac 12 & 1 & & & & \\[8pt] \frac 13 & \frac 12 & 1 & & & \\[8pt] \vdots & \vdots & \vdots & \ddots & & \\[8pt] \frac 1n & \frac 1{n-1} & \frac 1{n-2} & \cdots & 1 & \\[8pt] \vdots & \vdots & \vdots & \cdots & \vdots & \ddots \end{pmatrix}^{-1} \begin{pmatrix} \frac 12 \\[8pt] \frac 13 \\[8pt] \frac 14 \\[8pt] \vdots \\[8pt] \frac 1{n+1} \\[8pt] \vdots \end{pmatrix}. \end{equation}

\noindent  Kenter's proof uses induction, definite integrals, convergence of power series, and Abel's Theorem.  In this paper, we recast this statement using the language of Riordan matrices.  We exhibit another proof as well as a generalization.  Our main result is the

\begin{main} \label{main} Consider sequences $\{ a_0, \, a_1, \, \dots, \, a_n, \, \dots \}$, $\{ b_0, \, b_1, \, \dots, \, b_n, \, \dots \}$ and $\{ c_0, \, c_1, \, \dots, \, c_n, \, \dots \}$ of complex numbers such that $a_0, \, b_0, \, c_0 \neq 0$; as well as an integer exponent $d$.  Assume that 
\begin{itemize}
\item[(i)] the power series $a(x) = \sum_n a_n \, x^n$, $b(x) = \sum_n b_n \, x^n$, $c(x) = \sum_n c_n \, x^n$, and $b(x)^d $ are convergent in the interval $|x| < 1$; and
\item[(ii)] the following complex residue exists:
\begin{equation*} \text{Res}_{z = 0} \left[ \dfrac {a(z) \, b(z^{-1})^{d} \, c(z^{-1})}{z} \right] = \dfrac {1}{2 \pi i} \oint_{|z| = 1} a(z) \, b(z^{{-1}})^{d} \, c(z^{{-1}}) \, \dfrac {dz}{z}. \end{equation*}
\end{itemize}
Then the matrix product
\begin{equation*} \begin{pmatrix} a_0 & a_1 & a_2 & \cdots & a_n & \cdots \end{pmatrix} \begin{pmatrix} b_0 & & & & & \\[8pt] b_1 & b_0 & & & & \\[8pt] b_2 & b_1 & b_0 & & & \\[8pt] \vdots & \vdots & \vdots & \ddots & & \\[8pt] b_n & b_{n-1} & b_{n-2} & \cdots & b_0 & \\[8pt] \vdots & \vdots & \vdots & \cdots & \vdots & \ddots \end{pmatrix}^{d} \begin{pmatrix} c_0 \\[8pt] c_1 \\[8pt] c_2 \\[8pt] \vdots \\[8pt] c_n \\[8pt] \vdots \end{pmatrix} \end{equation*}
\noindent is equal to the above residue.
\end{main}

\noindent The infinite-dimensional lower triangular matrix is an example of a Riordan matrix.  Specifically, it is that Riordan matrix associated with the power series $b(x)^d$.  Kenter's result follows by careful analysis of the power series
\begin{equation} \begin{matrix} a(x) & = & - \dfrac {\log \, (1-x)}{x} & = & \displaystyle 1 + \frac 12 \, x + \frac 13 \, x^{2} + \cdots + \frac 1{n+1} \, x^{n} + \cdots \\[15pt] b(x)^{-1} & = & - \dfrac {x}{\log \, (1-x)} & = & \displaystyle 1 - \frac 12 \, x -\frac 1{12} \, x^{2} - \frac 1{24} \, x^{3} - \cdots - L_n \, x^n - \cdots \\[15pt] c(x) & = & \dfrac {a(x) - 1}{x} & = & \displaystyle \displaystyle \frac 12 + \frac 13 \, x + \frac 14 \, x^{2} + \cdots + \frac 1{n+2} \, x^{n} + \cdots \end{matrix} \end{equation}

\noindent The coefficients $L_n$ are sometimes called the ``logarithmic numbers'' or the ``Gregory coefficients''; these are basically the Bernoulli numbers of the second kind up to a choice of sign.  (Kenter employs the coefficients $c_k = - L_k$.)  The idea of this paper is that we have the matrix product
\begin{equation} \begin{pmatrix} 1 & & & & & \\[8pt] \frac 12 & 1 & & & & \\[8pt] \frac 13 & \frac 12 & 1 & & & \\[8pt] \vdots & \vdots & \vdots & \ddots & & \\[8pt] \frac 1n & \frac 1{n-1} & \frac 1{n-2} & \cdots & 1 & \\[8pt] \vdots & \vdots & \vdots & \cdots & \vdots & \ddots \end{pmatrix}^{-1} \begin{pmatrix} \frac 12 \\[8pt] \frac 13 \\[8pt] \frac 14 \\[8pt] \vdots \\[8pt] \frac 1{n+1} \\[8pt] \vdots \end{pmatrix} = \begin{pmatrix} \frac 12 \\[8pt] \frac 1{12} \\[8pt] \frac 1{24} \\[8pt] \vdots \\[8pt] L_n \\[8pt] \vdots \end{pmatrix}, \end{equation}
\noindent which is equivalent to the recursive identity $\sum_{m=0}^{n-1} L_m/(n-m) = 0$, that is valid whenever $n = 2, \, 3, \, 4, \, \dots$.  The matrix product, and hence the recursive identity, can be \emph{derived} from properties of Riordan matrices.  Kenter's result follows from the identity $\sum_{m=1}^{\infty} L_m / m = \gamma$, which in turn follows from an identity involving a definite integral.

As another consequence of our main result, we can also show that Euler's number
\begin{equation} e = \displaystyle \lim_{n \to \infty} \left( 1 + \dfrac {1}{n} \right)^n = 2.7182818284 \dots \end{equation}

\noindent can be represented as a product of an infinite-dimensional row vector, a lower triangular matrix, and an infinite-dimensional  column vector.

\begin{euler} \label{euler} For any integers $p$, $q$, and $d$ with $p \, q > 1$, the number
\begin{equation*} \dfrac {p \, q}{p \, q - 1} \, \sqrt[p]{e^d} = \displaystyle \lim_{n \to \infty} \left[ \dfrac {p \, q}{p \, q - 1}  \left( 1 + \dfrac {1}{p \, n} \right)^{d n} \right] \end{equation*}
\noindent is equal to the matrix product
\begin{equation*} \begin{pmatrix} 1 & \frac 1{p} & \frac 1{p^2} & \cdots & \frac 1{p^n} & \cdots \end{pmatrix} \begin{pmatrix} 1 & & & & & \\[8pt] \frac 1{1!} & 1 & & & & \\[8pt] \frac 1{2!} & \frac 1{1!} & 1 & & & \\[8pt] \vdots & \vdots & \vdots & \ddots & & \\[8pt] \frac 1{n!} & \frac 1{(n-1)!} & \frac 1{(n-2)!} & \cdots & 1 & \\[8pt] \vdots & \vdots & \vdots & \cdots & \vdots & \ddots \end{pmatrix}^{d} \begin{pmatrix} 1 \\[8pt] \frac 1{q} \\[8pt] \frac 1{q^2} \\[8pt] \vdots \\[8pt] \frac 1{q^n} \\[8pt] \vdots \end{pmatrix}. \end{equation*}
\end{euler}

In the process of proving these generalizations, we present a representation theoretic view of Riordan matrices.  That is, we consider the matrices as representations $\pi: G \to GL(V)$ of a certain group $G$ -- namely, the Riordan group -- acting on an infinite-dimensional vector space $V$ -- namely, the collection of those formal power series $h(x)$ in $\mathbb C \llbracket x \rrbracket$ where $h(0) = 0$.

\section{Introduction to Riordan Matrices}

We wish to list several key results in the theory of Riordan matrices.  To do so, we recast this theory using techniques from representation theory very much in the spirit of Bacher \cite{Bacher:2006p20934}.  Our ultimate goal in this section is to explain how Riordan matrices are connected to a permutation representation $\pi: G \to GL(V)$ of a certain group $G$ acting on an infinite dimensional vector space $V$.  Some of the notation in the sequel will differ from standard notation such as that give by Shapiro et al. \cite{Shapiro:1991p12411} and Sprugnoli \cite{Sprugnoli:1994p23729}, \cite{Sprugnoli:1995p23716}, but we will explain the connection.  

\subsection{Group Actions}

Before developing the representation theoretic view, we give the definition of a Riordan matrix and few related useful properties.  Let $k$ be a field; it is customary to set $k = \mathbb C$ as the set of complex numbers, but, in practice, $k = \mathbb Q$ is the set of rational numbers.  Set $k \llbracket x \rrbracket$ as the collection of formal power series in an indeterminate $x$; we will view this as a $k$-vector space with countable basis $\bigl \{ 1, \, x, \, x^2, \, \dots, \, x^n, \, \dots \bigr \}$.  For most of this article, we will not be concerned with regions of convergence for these series.

There are three binary operations $k \llbracket x \rrbracket \times k \llbracket x \rrbracket \to k \llbracket x \rrbracket$ which will be of importance to us, namely multiplication $\bullet$, composition $\circ$, and addition $+$.  Explicitly, if we write
\begin{equation}  f(x) = \sum_{n=0}^{\infty} f_n \, x^n \qquad \text{and} \qquad g(x) = \sum_{n=0}^{\infty} g_n \, x^n \end{equation}
\noindent then we have the formal power series
\begin{equation} \begin{aligned} \bigl( f \bullet g \bigr)(x) & = \sum_{n=0}^{\infty} \left[ \sum_{m=0}^n f_m \, g_{n-m} \right] x^n \\ \bigl( f \circ g \bigr)(x) & = \sum_{n=0}^{\infty} \left[ \sum_{m=0}^{\infty} f_m \left( \sum_{n_1 + \cdots + n_m = n} g_{n_1} \, \cdots \, g_{n_m} \right) \right] x^n \\ \bigl( f + g \bigr)(x) & = \sum_{n=0}^{\infty} \biggl[  f_n + g_n \biggr] x^n \end{aligned} \end{equation}

\noindent There are three subsets of the vector space $k \llbracket x \rrbracket$ which will be of interest to us in the sequel.

\begin{appell} \label{appell} Define the subsets
\begin{equation*} \begin{aligned} H & = \left \{ f(x) \in k \llbracket x \rrbracket \ \biggl| \ f(0) \neq 0 \right \} \\ K & = \left \{ g(x) \in k \llbracket x \rrbracket \ \biggl| \ \text{$g(0) = 0$ yet $g'(0) \neq 0$} \right \} \\ V & = \left \{ h(x) \in k \llbracket x \rrbracket \ \biggl| \ h(0) = 0 \right \} \end{aligned} \end{equation*}
\begin{enumerate}
\item[(i)] $H$ is a group under multiplication $\bullet$, $K$ is a group under composition $\circ$, and $V$ is a group under addition $+$.  In particular, $V$ is a $k$-vector space with countable basis $\bigl \{ x, \, x^2, \, \dots, \, x^n, \, \dots \bigr \}$.
\item[(ii)] The map $\varphi: K \to \text{Aut}(H)$ which sends $g(x) \in K$ to the automorphism $\varphi_g: f(x) \mapsto \bigl( f \circ \overline {g} \bigr)(x)$ is a group homomorphism, where $\overline {g}(x)$ is the compositional inverse of $g(x)$.  In particular, $G = H \rtimes_{\varphi} K$ is a group under the binary operation ${\ast}: G \times G \to G$ defined by
\begin{equation*} \bigl( f_1, \, g_1 \bigr) {\ast} \bigl( f_2, \, g_2 \bigr) = \bigl( f_1 \bullet \varphi_{g_1}(f_2), \ g_1 \circ g_2 \bigr). \end{equation*}
\item[(iii)] The map ${\ast}: G \times V \to V$ defined by $\bigl(f, \, g \bigr) {\ast} h = f \bullet \bigl( h \circ \overline {g} \bigr)$ is a group action of $G$ on $V$. 
\end{enumerate}
\end{appell}

We use $\overline g(x)$ to denote the compositional inverse $g^{-1}(x)$ so that we will not confuse this with the multiplicative inverse $g(x)^{-1}$.  Later, we will show that $G$ is isomorphic to the Riordan group $\mathbf R$.  Moreover, we will show that $H$, a normal subgroup of $G$, is isomorphic to the Appell subgroup of $\mathbf R$.  The motivation of this result is to use the action of $G$ on $V$ to write down a permutation representation $\pi: G \to GL(V)$, then use the canonical basis $\bigl \{ x, \, x^2, \, \dots, \, x^n, \, \dots \bigr \}$ of $V$ to list infinite-dimensional matrices.

\begin{proof} We show (i) to fix some notation to be used in the sequel.  Since $\bigl( f \bullet g \bigr)(0) = f(0) \, g(0) \neq 0$ for any $f(x), \, g(x) \in H$, we see that $\bullet: H \times H \to H$ is an associative binary operation.  The identity is the constant power series $e(x) = 1$, and the inverse of $f(x)$ is its reciprocal, seen to be a power series by expressing as a formal geometric series:
\begin{equation} \begin{aligned} \dfrac {1}{f(x)} & = \dfrac {1}{f_0} \cdot \dfrac {1}{\displaystyle 1 - \sum_{n=0}^{\infty} (-f_n/f_0) \, x^n} \\ & = \sum_{n=0}^{\infty} \left[ \sum_{m=0}^{\infty} \sum_{n_1 + \cdots + n_m = n-m} (-1)^m \, \dfrac {f_{n_1+1} \, \cdots \, f_{n_m+1}}{{f_0}^{m+1}} \right] x^n. \end{aligned} \end{equation}

\noindent Since $\bigl( f \circ g \bigr)(0) = f \bigl( g(0) \bigr) = f(0) = 0$ and $\bigl( f \circ g \bigr)'(0) = f' \bigl( g(0) \bigr) \, g'(0) = f'(0) \, g'(0) \neq 0$ for any $f(x), \, g(x) \in K$, we see that $\circ: K \times K \to K$ is an associative binary operation.  The identity is the power series $\text{id}(x) = x$, and the inverse of $g(x)$ is its compositional inverse $\overline {g}(x) = \sum_n \overline {g}_n \, x^n$ having the implicitly defined coefficients
\begin{equation} \begin{aligned} & \overline {g}_0 = 0, \\ & \overline {g}_1 = \dfrac {1}{g_1}, \\ \sum_{m=0}^{n} & \overline {g}_m \left[ \sum_{n_1 + \cdots + n_m = n} g_{n_1} \, \cdots \, g_{n_m} \right] = 0 \qquad \text{for $n = 2, \, 3, \, \dots$.} \end{aligned} \end{equation}

\noindent Since $\bigl( f + g \bigr)(0) = f(0) + g(0) = 0$ for any $f(x), \, g(x) \in V$, we see that $+: V \times V \to V$ is an associative binary operation.  The identity is the constant power series $o(x) = 0$, and the inverse of $h(x)$ is the negation $-h(x)$, seen to be a power series with $\bigl( -h \bigr)(0) = -h(0) = 0$.

Now we show (ii). Since $\bigl( f \circ \overline {g} \bigr)(0) = f \bigl( \overline {g}(0) \bigr) = f(0) \neq 0$ for any $f(x) \in H$ and $g(x) \in K$, we see that $\varphi: K \to \text{Aut}(H)$ is well-defined.  Given $g(x), \, h(x) \in K$ we have $\varphi_g \circ \varphi_h = \varphi_{g \circ h}$ because for all $f(x) \in H$ we have
\begin{equation} \begin{aligned} \biggl( \varphi_g \circ \varphi_h \biggr) \bigl[ f(x) \bigr] & = \varphi_g \biggl[ \bigl( f \circ \overline {h} \bigr)(x) \biggr] \\ & = \biggl( f \circ \overline {h} \circ \overline {g} \biggr)(x) = \biggl( f \circ \overline {g \circ h} \biggr)(x) \\ & = \varphi_{g \circ h} \bigl[ f(x) \bigr]. \end{aligned} \end{equation}

\noindent Hence $\varphi: K \to \text{Aut}(H)$ is indeed a group homomorphism.  The semi-direct product $G = H \rtimes_{\varphi} K$ consists of pairs $\bigl( f(x), \, g(x) \bigr)$ with $f(x) \in H$ and $g(x) \in K$, where the binary operation ${\ast}: G \times G \to G$ is defined by 
\begin{equation} \biggl( f_1(x), \ g_1(x) \biggr) {\ast} \biggl( f_2(x), \ g_2(x) \biggr) = \biggl( f_1(x) \, f_2 \bigl( \overline {g_1}(x) \bigr),  \ g_1 \bigl( g_2(x) \bigr) \biggr). \end{equation}

Finally, we show (iii). The map ${\ast}: G \times V \to V$ is defined as the formal identity
\begin{equation} \biggl(f(x), \ g(x) \biggr) {\ast} h(x) = f(x) \, h \bigl( \overline {g}(x) \bigr). \end{equation}

\noindent Since $\bigl[ \bigl(f, \, g \bigr) {\ast} h \bigr](0) = f(0) \ h \bigl( \overline {g}(0) \bigr) = f(0) \, h(0) = 0$, we see that the map ${\ast}: G \times V \to V$ is well-defined.  As the identity element of $G$ is $\bigl( e(x), \, \text{id}(x) \bigr) = \bigl( 1, \, x \bigr)$, we see that $\bigl( e(x), \, \text{id}(x) \bigr) {\ast} h(x) = h(x)$ so that it acts trivially on $V$. Given two elements $\bigl( f_1, \, g_1 \bigr), \, \bigl( f_2, \, g_2 \bigr) \in G$ and $h(x) \in V$, we have the identity
\begin{equation} \begin{aligned} \biggl( f_1(x), \, g_1(x) \biggr) & {\ast} \biggl[  \biggl( f_2(x), \, g_2(x) \biggr)  {\ast} h(x) \biggr] \\ & = \biggl( f_1(x), \, g_1(x) \biggr) {\ast} \biggl[  f_2(x) \ h \bigl( \overline {g_2}(x) \bigr) \biggr] \\ & = f_1(x) \ f_2 \bigl( \overline {g_1}(x) \bigr) \ h \bigl( \overline {g_2} \circ \overline {g_1}(x) \bigr) \\ & = f_1(x) \ f_2 \bigl( \overline {g_1}(x) \bigr) \ h \bigl( \overline {g_1 \circ g_2}(x) \bigr)  \\ & = \biggl( f_1(x) \, f_2 \bigl( \overline {g_1}(x) \bigr),  \ g_1 \bigl( g_2(x) \bigr) \biggr) {\ast} h(x) \\ & = \biggl[ \biggl( f_1(x), \, g_1(x) \biggr) {\ast} \biggl( f_2(x), \, g_2(x) \biggr) \biggr] {\ast} h(x). \end{aligned} \end{equation}

\noindent Similarly, given two elements $h_1(x), \, h_2(x) \in V$ and $\bigl( f, \, g \bigl) \in G$, we have the identity
\begin{equation} \begin{aligned} \biggl( f(x), \ g(x) \biggr) & {\ast} \biggl[ h_1(x) + h_2(x) \biggr] \\ & = f(x) \, \biggl[ h_1 \bigl( \overline {g}(x) \bigr) + h_2 \bigl( \overline {g}(x) \bigr) \biggr] \\ & = \biggl( f(x), \ g(x) \biggr) {\ast} h_1(x) + \biggl( f(x), \ g(x) \biggr) {\ast} h_2(x). \end{aligned} \end{equation}
\noindent Hence ${\ast}: G \times V \to V$ is indeed a group action.  \end{proof}

\subsection{Riordan Matrices.}  Recall that the set
\begin{equation} V = \left \{ h(x) \in k \llbracket x \rrbracket \ \biggl| \ h(0) = 0 \right \} \end{equation}

\noindent is a $k$-vector space with countable basis $\bigl \{ x, \, x^2, \, \dots, \, x^n, \, \dots \bigr \}$.  Since the semi-direct product $G = H \rtimes_{\varphi} K$ acts on $V$, we have a ``permutation'' representation $\pi: G \to GL(V)$.  Explicitly, this representation is defined on the basis elements of $V$ via the formal identity
\begin{equation} \begin{aligned} \biggl( f(x), \, g(x) \biggr) {\ast} \, x^m & = f(x) \, \bigl[ \overline {g}(x) \bigr]^m \\ & = \sum_{n=1}^{\infty} l_{n,m} \, x^n \qquad \text{for $m = 1, \, 2, \, 3, \, \dots$.} \end{aligned} \end{equation}

\noindent (Recall that $\overline {g}(x)$ is the compositional inverse of $g(x)$.) The matrix with respect to the basis $\bigl \{ x, \, x^2, \, \dots, \, x^n, \, \dots \bigr \}$ is given by the lower triangular matrix
\begin{equation} \pi \biggl( f(x), \ g(x) \biggr) = \begin{pmatrix} l_{1,1} & & & & & \\[5pt] l_{2,1} & l_{2,2} & & & & \\[5pt] l_{3,1} & l_{3,2} & l_{3,3} & & & \\[5pt] \vdots & \vdots & \vdots  & \ddots & & \\[5pt] l_{n,1} & l_{n,2} & l_{n,3} & \cdots & l_{n,n} & \\[5pt] \vdots & \vdots & \vdots & \cdots & \vdots & \ddots \end{pmatrix}. \end{equation}

\noindent Recall that $g(0) = 0$ yet $f(0), \, g'(0) \neq 0$.  The following result explains the main multiplicative property of these matrices.

\begin{product} \label{product} Continue notation as above.
\begin{itemize}
\item[(i)] $\pi: G \to GL(V)$ is a group homomorphism.  That is,
\begin{equation*} \pi \biggl( f_1(x), \ g_1(x) \biggr) \ \pi \biggl( f_2(x), \ g_2(x) \biggr) = \pi \biggl( f_1(x) \, f_2 \bigl( \overline {g_1}(x) \bigr),  \ g_1 \bigl( g_2(x) \bigr) \biggr). \end{equation*}
\item[(ii)] For a generating function $t(x) = t_0 + t_1 \, x + t_2 \, x^2 + \cdots$ with $t_0 \neq 0$,
\begin{equation*} \pi \biggl( f(x), \ g(x) \biggr) \ \pi \biggl( t(x), \ \text{id}(x) \biggr) = \begin{pmatrix} \displaystyle \sum_{p=1}^{m} l_{n,p} \ t_{p-m} \end{pmatrix}_{n,m \geq 1}. \end{equation*} 
\end{itemize}
\end{product}

Such matrices $\pi \bigl( f, \, g \bigr)$ are called the \textit{Riordan matrices associated to the pair $\bigl( f, \, g \bigr)$}.  The collection $\mathbf R$ of Riordan matrices is a group which is isomorphic to $G = H \rtimes_{\varphi} K$; this is the \textit{Riordan group}.   The collection of matrices $\pi\bigl( f, \, \text{id} \bigr)$ is a group which is isomorphic to $H$; this normal subgroup is the \textit{Appell subgroup of $\mathbf{R}$}.  
 
\begin{proof} We show (i).  In the proof of Proposition \ref{appell}, we found that for each $h(x) \in V$ we have the following formal identity involving power series as elements of $k \llbracket x \rrbracket$:
\begin{equation} \begin{aligned} \biggl( f_1(x), \, g_1(x) \biggr) & {\ast} \biggl[  \biggl( f_2(x), \, g_2(x) \biggr)  {\ast} h(x) \biggr] \\ & = \biggl[ \biggl( f_1(x), \, g_1(x) \biggr) {\ast} \biggl( f_2(x), \, g_2(x) \biggr) \biggr] {\ast} h(x) \\ & = \biggl( f_1(x) \, f_2 \bigl( \overline {g_1}(x) \bigr),  \ g_1 \bigl( g_2(x) \bigr) \biggr) \ast h(x). \end{aligned} \end{equation}
\noindent In particular, this holds for the basis elements $h(x) = x^n$, so the result follows. 

Now we show (ii). For a generating function $t(x) = t_0 + t_1 \, x + t_2 \, x^2 + \cdots$, we have the product
\begin{equation} \biggl( t(x), \ \text{id}(x) \biggr) {\ast} x^m = t(x) \, x^m = \sum_{n=1}^{\infty} t_{n-m} \, x^{n}; \end{equation}
\noindent so that matrices in the Appell subgroup are in the form
\begin{equation} \pi \biggl( t(x), \ \text{id}(x) \biggr) = \begin{pmatrix} t_0 & & & & & \\[5pt] t_1 & t_0 & & & & \\[5pt] t_2 & t_1 & t_0 & & & \\[5pt] \vdots & \vdots & \vdots & \ddots & & \\[5pt] t_{n-1} & t_{n-2} & t_{n-3} & \cdots & t_0 & \\[5pt] \vdots & \vdots & \vdots & \cdots & \vdots & \ddots \end{pmatrix}. \end{equation}

\noindent This gives the matrix product
\begin{equation} \begin{aligned} \pi \biggl( f(x), \ g(x) \biggr) \ \pi \biggl( t(x), \ x \biggr) & = \begin{pmatrix} \biggl. l_{n,p} \end{pmatrix}_{n,p \geq 1} \begin{pmatrix} \biggl. t_{p-m} \end{pmatrix}_{p,m \geq 1} \\ & = \begin{pmatrix} \displaystyle \sum_{p=1}^{m} l_{n,p} \ t_{p-m} \end{pmatrix}_{n,m \geq 1} \end{aligned} \end{equation}

\noindent so the result follows. \end{proof}

\subsection{Examples}

Let $k = \mathbb Q$.  Using elementary Calculus, we find the power series expansions
\begin{equation} \begin{aligned} - \frac{\ln \, (1-x)}{x} & = 1 + \frac 12 \, x + \frac 13 \, x^{2} + \frac 14 \, x^{3} + \frac 15 \, x^{4} + \frac 16 \, x^{5} + \cdots, \\[5pt] - \frac{x}{\ln \, (1-x)} & = 1 - \frac 12 \, x -\frac 1{12} \, x^{2} - \frac 1{24} \, x^{3} - \frac {19}{720} \, x^{4} - \frac{3}{160} \, x^{5} + \cdots; \end{aligned} \end{equation}

\noindent which are valid on whenever $|x| < 1$.  Hence the formal power series
\begin{equation} f(x) = 1 + \frac 12 \, x + \frac 13 \, x^{2} + \frac 14 \, x^{3} + \cdots + \frac 1{n+1} \, x^{n} + \cdots \end{equation}
\noindent is an element of $H$, and has multiplicative inverse
\begin{equation} \dfrac {1}{f(x)} = 1 - \frac 12 \, x -\frac 1{12} \, x^{2} - \frac 1{24} \, x^{3} - \frac {19}{720} \, x^{4} - \frac{3}{160} \, x^{5} + \cdots. \end{equation}

\noindent We have the product
\begin{equation} \biggl( f(x), \ \text{id}(x) \biggr) {\ast} x^m = f(x) \, x^m = \sum_{n=1}^{\infty} \dfrac 1{n-m +1} \, x^{n} \end{equation}

\noindent which yields the matrix
\begin{equation} \pi \bigl( f, \, \text{id} \bigr) = \begin{pmatrix} 1 & & & & & \\[8pt] \frac 12 & 1 & & & & \\[8pt] \frac 13 & \frac 12 & 1 & & & \\[8pt] \vdots & \vdots & \vdots & \ddots & & \\[8pt] \frac 1{n} & \frac 1{n-1} & \frac 1{n-2} & \cdots & 1 & \\[8pt] \vdots & \vdots & \vdots & \cdots & \vdots & \ddots \end{pmatrix}. \end{equation}

\noindent Similarly, we have the product
\begin{equation} \begin{aligned} \biggl( \dfrac 1{f(x)}, & \ \text{id}(x) \biggr) {\ast} x^m = \dfrac 1{f(x)} \, x^m \\ & = x^m - \frac 12 \, x^{m+1} -\frac 1{12} \, x^{m+2} - \frac 1{24} \, x^{m+3}  - \frac {19}{720} \, x^{m+4} + \cdots . \end{aligned} \end{equation}

\noindent Since we may use Theorem \ref{product} to conclude that $\pi \bigl( f, \, \text{id} \bigr)^{-1} = \pi \bigl( 1/f, \, \text{id} \bigr)$, we find the identity
\begin{equation} \begin{pmatrix} 1 & & & & & \\[8pt] \frac 12 & 1 & & & & \\[8pt] \frac 13 & \frac 12 & 1 & & & \\[8pt] \vdots & \vdots & \vdots & \ddots & & \\[8pt] \frac 1{n} & \frac 1{n-1} & \frac 1{n-2} & \cdots & 1 & \\[8pt] \vdots & \vdots & \vdots & \cdots & \vdots & \ddots \end{pmatrix}^{-1} = \begin{pmatrix} 1 & & & & & \\[8pt] -\frac 12 & 1 & & & & \\[8pt] -\frac 1{12} & -\frac 12 & 1 & & & \\[8pt] - \frac 1{24} & - \frac 1{12} & - \frac 12 & 1 & & \\[8pt] - \tfrac {19}{720} & - \frac 1{24} & - \frac 1{12} & - \frac 12 & 1 & \\[8pt] \vdots & \vdots & \vdots & \cdots & \vdots & \ddots \end{pmatrix}. \end{equation}

\noindent These matrices are elements of the Appell subgroup of $\mathbf{R}$.

\subsection{Relation with Standard Notation}

Standard references for Riordan matrices are Shapiro et al. \cite{Shapiro:1991p12411} and Sprugnoli \cite{Sprugnoli:1994p23729}, \cite{Sprugnoli:1995p23716}.  The notation $\pi \bigl( f, \, g \bigr)$ employed above is not the typical one, so we explain the connection.  Consider sequences $\bigl \{ G_0, \, G_1, \, G_2, \, \dots, \, G_n, \, \dots \bigr \}$ and $\bigl \{ F_1, \, F_2, \, F_3, \, \dots, \, F_n, \, \dots \bigr \}$ of complex numbers $k = \mathbb C$, where $G_0, \, F_1 \neq 0$.  Upon associating generating functions $G(x) = G_0 + G_1 \, x + G_2 \, x^2 + \cdots$ and $F(x) = F_1 \, x + F_2 \, x^2 + F_3 \, x^3 + \cdots$ to these sequences respectively, the standard notation for a Riordan matrix is that infinite-dimensional matrix given by
\begin{equation} L = \bigl[ G(x), \ F(x) \bigr] = \pi \bigl( G(x), \ \overline {F}(x) \bigr) = \begin{pmatrix} \biggl. l_{n,m} \end{pmatrix}_{n,m \geq 1} \end{equation}
\noindent in terms of the compositional inverse $\overline {F}(x)$ of $F(x)$.  Indeed, the entry $l_{n,m}$ in the $n$th row and $m$th column satisfies the relation
\begin{equation} G(x) \, \bigl[ F(x) \bigr]^m = \sum_{n=1}^{\infty} l_{n,m} \, x^n \qquad \text{for $m = 1, \, 2, \, 3, \, \dots$} \end{equation}

\noindent as formal power series in $\mathbb C \llbracket x \rrbracket$. Equivalently, a Riordan matrix $L$ can be defined by a pair $\bigl( G(x), \, F(x) \bigr)$ of generating functions.

\begin{standard}[Fundamental Theorem of the Riordan Group: \cite{Nkwanta:2005p20729}, \cite{Shapiro:1991p12411}, \cite{Sprugnoli:1995p23716}] Continue notation as above.
\begin{itemize}
\item[(i)] The product of Riordan matrices is again a Riordan matrix.  Explicitly, their product satisfies the relation 
\begin{equation*} \bigl[ G_1(x), \ F_1(x) \bigr] \ \bigl[ G_2(x), \ F_2(x) \bigr] = \biggl[ G_1(x) \, G_2 \bigl( F_1(x) \bigr), \ F_2 \bigl( F_1(x) \bigr) \biggr]. \end{equation*}
\item[(ii)] For a generating function $T(x) = T_0 + T_1 \, x + T_2 \, x^2 + \cdots$ with $T_0 \neq 0$, we have the product
\begin{equation*}\bigl[ G(x), \ F(x) \bigr] \ \bigl[ T(x), \ x \bigr] = \begin{pmatrix} \displaystyle \sum_{p=1}^{m} l_{n,p} \ T_{p-m} \end{pmatrix}_{n,m \geq 1} \end{equation*} 
\end{itemize} \end{standard}

\begin{proof} Statement (i) is shown in \cite[Eq. 5]{Shapiro:1991p12411} and \cite[Proof of Thm. 2.1]{Nkwanta:2005p20729}, but we give an alternate proof.  Upon denoting $f_i(x) = G_i(x)$ and $g_i(x) = \overline {F}_i(x)$ for $i = 1$ and $2$, we find the matrix product
\begin{equation} \begin{aligned} \bigl[ G_1(x), \ F_1(x) \bigr] & \ \bigl[ G_2(x), \ F_2(x) \bigr] \\ & = \pi \biggl( f_1(x), \ g_1(x) \biggr) \ \pi \biggl( f_2(x), \ g_2(x) \biggr) \\ & = \pi \biggl( f_1(x) \, f_2 \bigl( \overline {g_1}(x) \bigr),  \ g_1 \bigl( g_2(x) \bigr) \biggr) \\ & = \biggl[ G_1(x) \, G_2 \bigl( F_1(x) \bigr), \ F_2 \bigl( F_1(x) \bigr) \biggr] \end{aligned} \end{equation}

\noindent which follows directly from Theorem \ref{product}. Statement (ii) is also shown in \cite{Nkwanta:2005p20729}, but it follows directly from Theorem \ref{product} as well. \end{proof}

\section{Proof of Kenter's Result and Generalizations}

\subsection{Main Result}

We now prove the following:

\vskip 0.1in \noindent \textbf{Theorem \ref{main}.} \emph{Consider sequences $\{ a_0, \, a_1, \, \dots, \, a_n, \, \dots \}$, $\{ b_0, \, b_1, \, \dots, \, b_n, \, \dots \}$ and $\{ c_0, \, c_1, \, \dots, \, c_n, \, \dots \}$ of complex numbers such that $a_0, \, b_0, \, c_0 \neq 0$; as well as an integer exponent $d$.  Assume that 
\begin{itemize}
\item[(i)] the power series $a(x) = \sum_n a_n \, x^n$, $b(x) = \sum_n b_n \, x^n$, $c(x) = \sum_n c_n \, x^n$, and $b(x)^ d $ are convergent in the interval $|x| < 1$; and
\item[(ii)] the following complex residue exists:
\begin{equation*} \text{Res}_{z = 0} \left[ \dfrac {a(z) \, b(z^{-1})^{d} \, c(z^{-1})}{z} \right] = \dfrac {1}{2 \pi i} \oint_{|z| = 1} a(z) \, b(z^{{-1}})^{d} \, c(z^{{-1}}) \, \dfrac {dz}{z}. \end{equation*}
\end{itemize}
Then the matrix product
\begin{equation*} \begin{pmatrix} a_0 & a_1 & a_2 & \cdots & a_n & \cdots \end{pmatrix} \begin{pmatrix} b_0 & & & & & \\[8pt] b_1 & b_0 & & & & \\[8pt] b_2 & b_1 & b_0 & & & \\[8pt] \vdots & \vdots & \vdots & \ddots & & \\[8pt] b_n & b_{n-1} & b_{n-2} & \cdots & b_0 & \\[8pt] \vdots & \vdots & \vdots & \cdots & \vdots & \ddots \end{pmatrix}^{d} \begin{pmatrix} c_0 \\[8pt] c_1 \\[8pt] c_2 \\[8pt] \vdots \\[8pt] c_n \\[8pt] \vdots \end{pmatrix} \end{equation*}
\noindent is equal to the above residue.}

\begin{proof} With the three power series $a(x) = \sum_n a_n \, x^n$, $b(x) = \sum_n b_n \, x^n$, and $c(x) = \sum_n c_n \, x^n$ convergent in the interval $|x| < 1$, consider the power series
\begin{equation} f(x) = b(x)^d \, c(x) = \sum_{n=0}^{\infty} f_n \, x^n \qquad \text{where $|x| < 1$.} \end{equation}

\noindent As elements of the Appell subgroup of $\mathbf{R}$, we invoke Theorem \ref{product} to see that we have the matrix product $\pi \bigl( f(x), \ x \bigr) = \pi \bigl( b(x), \ x \bigr)^{d} \ \pi \bigl( c(x), \ x \bigr)$.  In particular, the first column is given by
\begin{equation} \begin{pmatrix} f_0 \\[8pt] f_1 \\[8pt] f_2 \\[8pt] \vdots \\[8pt] f_n \\[8pt] \vdots \end{pmatrix} = \begin{pmatrix} b_0 & & & & & \\[8pt] b_1 & b_0 & & & & \\[8pt] b_2 & b_1 & b_0 & & & \\[8pt] \vdots & \vdots & \vdots & \ddots & & \\[8pt] b_n & b_{n-1} & b_{n-2} & \cdots & b_0 & \\[8pt] \vdots & \vdots & \vdots & \cdots & \vdots & \ddots \end{pmatrix}^{d} \begin{pmatrix} c_0 \\[8pt] c_1 \\[8pt] c_2 \\[8pt] \vdots \\[8pt] c_n \\[8pt] \vdots \end{pmatrix}. \end{equation}

\noindent Hence the matrix product 
\begin{equation} \label{nkwanta} \begin{pmatrix} a_0 & a_1 & a_2 & \cdots & a_n & \cdots \end{pmatrix} \begin{pmatrix} b_0 & & & & & \\[8pt] b_1 & b_0 & & & & \\[8pt] b_2 & b_1 & b_0 & & & \\[8pt] \vdots & \vdots & \vdots & \ddots & & \\[8pt] b_n & b_{n-1} & b_{n-2} & \cdots & b_0 & \\[8pt] \vdots & \vdots & \vdots & \cdots & \vdots & \ddots \end{pmatrix}^{d} \begin{pmatrix} c_0 \\[8pt] c_1 \\[8pt] c_2 \\[8pt] \vdots \\[8pt] c_n \\[8pt] \vdots \end{pmatrix} \end{equation}
\noindent is equal to the sum $\sum_n a_n \, f_n$.  We wish to evaluate this sum using complex analysis.

By assumption, the power series $a(x)$, $b(x)$, and $c(x)$ are convergent in the interval $|x| < 1$.  Hence, for each fixed real number $r$ satisfying $0 < r < 1$, the functions $a(z)$ and $f(z)$ are \emph{uniformly} convergent inside a closed disk $|z| \leq r$.  Hence we can interchange summation and integration to find the integral around the boundary to be equal to
\begin{equation} \begin{aligned}
\dfrac {1}{2 \pi i} \oint_{|z| = r} & a(z) \, b(z^\ast)^d \, c(z^\ast) \, \dfrac {dz}{z} \\ & = \dfrac {1}{2 \pi i} \oint_{|z| = r} a(z) \, f(z^\ast) \, \dfrac {dz}{z} \\ & = \sum_{n_1=0}^{\infty} \sum_{n_2=0}^{\infty} a_{n_1} \, f_{n_2} \, r^{n_1+n_2} \cdot \dfrac {1}{2 \pi} \int_0^{2 \pi} e^{i (n_1-n_2) \theta} \, d \theta \\ & = \sum_{n=0}^{\infty} a_n \, f_n \, r^{2n}. \end{aligned} \end{equation}

\noindent Here $z^\ast$ is the complex conjugate of $z$.  As $r \to 1$, the integral exists so by Cauchy's Residue Theorem it must be equal to
\begin{equation} \begin{aligned} 
\text{Res}_{z = 0} & \left[ \dfrac {a(z) \, b(z^{-1})^{d} \, c(z^{-1})}{z} \right] \\ & = \lim_{r \to 1} \left[ \dfrac {1}{2 \pi i} \oint_{|z| = 1} a(z) \, b(z^{{-1}})^{d} \, c(z^{{-1}}) \, \dfrac {dz}{z} \right] \\ & = \lim_{r \to 1} \left[ \sum_{n=0}^{\infty} a_n \, f_n \, r^{2n} \right] \\ & = \sum_{n=0}^{\infty} a_n \, f_n. \end{aligned} \end{equation}

\noindent The Theorem follows upon equating this with equation \eqref{nkwanta}. \end{proof}

\subsection{Applications}

We explain how to use Theorem \ref{main} in order to express Euler's number $e = 2.7182818284 \dots $ in terms of Riordan matrices.

\vskip 0.1in \noindent \textbf{Corollary \ref{euler}.} \emph{For any integers $p$, $q$, and $d$ with $p \, q > 1$, the number
\begin{equation*} \dfrac {p \, q}{p \, q - 1} \, \sqrt[p]{e^d} = \displaystyle \lim_{n \to \infty} \left[ \dfrac {p \, q}{p \, q - 1}  \left( 1 + \dfrac {1}{p \, n} \right)^{d n} \right] \end{equation*}
\noindent is equal to the matrix product
\begin{equation*} \begin{pmatrix} 1 & \frac 1{p} & \frac 1{p^2} & \cdots & \frac 1{p^n} & \cdots \end{pmatrix} \begin{pmatrix} 1 & & & & & \\[8pt] \frac 1{1!} & 1 & & & & \\[8pt] \frac 1{2!} & \frac 1{1!} & 1 & & & \\[8pt] \vdots & \vdots & \vdots & \ddots & & \\[8pt] \frac 1{n!} & \frac 1{(n-1)!} & \frac 1{(n-2)!} & \cdots & 1 & \\[8pt] \vdots & \vdots & \vdots & \cdots & \vdots & \ddots \end{pmatrix}^{d} \begin{pmatrix} 1 \\[8pt] \frac 1{q} \\[8pt] \frac 1{q^2} \\[8pt] \vdots \\[8pt] \frac 1{q^n} \\[8pt] \vdots \end{pmatrix}. \end{equation*}}

\begin{proof} The coefficients of the matrices above correspond to the three power series
\begin{equation} \left. \begin{matrix} a(x) & = & \dfrac {1}{1 - x/p} & = & \displaystyle \sum_{n=0}^{\infty} \dfrac {x^n}{p^n} \\[15pt] b(x) & = & e^x & = & \displaystyle \sum_{n=0}^{\infty} \dfrac {x^n}{n!} \\[15pt] c(x) & = & \dfrac {1}{1 - x/q} & = & \displaystyle \sum_{n=0}^{\infty} \dfrac {x^n}{q^n} \end{matrix} \right \} \qquad \text{where $|x| < 1$.} \end{equation}
\noindent For a complex number $z$ with $|z| < 1$, we have the identity
\begin{equation} \begin{aligned} \dfrac {a(z) \, b(z^{-1})^{d} \, c(z^{-1})}{z} & = \biggl[ \dfrac 1{1 - z/p} \biggr] \, \biggl[ e^{d \, z^{-1}} \biggr] \, \biggl[ \dfrac 1{1 - z^{-1} / q} \biggr] \\ & = \sum_{n= -\infty}^{\infty} \left[ \sum_{n_1 - n_2 - n_3 = n+1} \dfrac {d^{n_2}}{p^{n_1} \, n_2! \, q^{n_3}} \right] z^n \end{aligned} \end{equation}
\noindent The residue corresponds to the coefficient of the $z^{-1}$ term, so we consider the terms where $n = -1$:
\begin{equation} \begin{aligned} \text{Res}_{z = 0} \left[ \dfrac {a(z) \, b(z^{-1})^{d} \, c(z^{-1})}{z} \right] & = \sum_{n_1 = n_2 + n_3} \dfrac {d^{n_2}}{p^{n_1} \, n_2! \, q^{n_3}} \\ & = \left[ \sum_{n_2 = 0}^{\infty} \dfrac {1}{n_2!} \left( \dfrac {d}{p} \right)^{n_2} \right] \left[ \sum_{n_3 = 0}^{\infty} \dfrac {1}{(p \, q)^{n_3}} \right] \\ &  = e^{d/p} \, \dfrac {p \, q}{p \, q - 1}. \end{aligned} \end{equation}
\noindent The Corollary follows now from Theorem \ref{main}. \end{proof}

Kenter's result is also an application of Theorem \ref{main}.

\begin{kenter}[\cite{Kenter:1999p20766}] The Euler-Mascheroni constant
\begin{equation*} \gamma = \displaystyle \lim_{n \to \infty} \left[ \left( \sum_{m=1}^n \dfrac 1m \right) - \ln n \right] = 0.5772156649 \dots \end{equation*}
\noindent is equal to the matrix product
\begin{equation*} \begin{pmatrix} 1 & \frac 12 & \frac 13 & \cdots & \frac 1n & \cdots \end{pmatrix} \begin{pmatrix} 1 & & & & & \\[8pt] \frac 12 & 1 & & & & \\[8pt] \frac 13 & \frac 12 & 1 & & & \\[8pt] \vdots & \vdots & \vdots & \ddots & & \\[8pt] \frac 1n & \frac 1{n-1} & \frac 1{n-2} & \cdots & 1 & \\[8pt] \vdots & \vdots & \vdots & \cdots & \vdots & \ddots \end{pmatrix}^{-1} \begin{pmatrix} \frac 12 \\[8pt] \frac 13 \\[8pt] \frac 14 \\[8pt] \vdots \\[8pt] \frac 1{n+1} \\[8pt] \vdots \end{pmatrix}.  \end{equation*}
\end{kenter}

\begin{proof} The coefficients of the matrices above correspond to the three power series
\begin{equation} \left. \begin{matrix} a(x) & = & - \dfrac {\log \, (1-x)}{x} & = & \displaystyle \sum_{n=0}^{\infty} \dfrac {x^n}{n+1} \\[15pt] b(x) & = & - \dfrac {\log \, (1-x)}{x} & = & \displaystyle \sum_{n=0}^{\infty} \dfrac {x^n}{n+1} \\[15pt] c(x) & = & \dfrac {a(x) - 1}{x} & = & \displaystyle \sum_{n=0}^{\infty} \dfrac {x^n}{n+2} \end{matrix} \right \} \qquad \text{where $|x| < 1$.} \end{equation}
\noindent We will choose the exponent $d = -1$.  We will express the reciprocal as the power series
\begin{equation} \begin{aligned} \dfrac {x}{\log \, (1-x)} & = -1 + \frac 12 \, x  + \frac 1{12} \, x^{2} + \frac 1{24} \, x^{3} + \frac {19}{720} \, x^{4} + \frac{3}{160} \, x^{5} + \cdots \\ & = \sum_{n=0}^{\infty} L_n \, x^n \end{aligned} \end{equation}
\noindent which is also convergent in the interval $|x| < 1$.  (Recall that the coefficients $L_n$ are sometimes called the ``logarithmic numbers'' or the ``Gregory coefficients''.)  For a complex number $z$ with $|z| < 1$, we have the identity
\begin{equation} \begin{aligned} & \dfrac {a(z) \, b(z^{-1})^{d} \, c(z^{-1})}{z} \\ & \qquad = - \dfrac {\log (1-z)}{z} + \left[ - \dfrac {\log \, (1-z)}{z} \right] \cdot \left[ \dfrac {z^{-1}}{\log \, (1-z^{-1})} \right] \\ & \qquad = \sum_{n=0}^{\infty} \dfrac {z^n}{n+1} + \sum_{n=-\infty}^{\infty} \left[ \sum_{m = -n}^{\infty} \dfrac {L_m}{n + m + 1} \right] z^n. \end{aligned} \end{equation}
\noindent The residue corresponds to the coefficient of the $z^{-1}$ term, so we consider the terms where $n = -1$:
\begin{equation} \begin{aligned} \text{Res}_{z = 0} \left[ \dfrac {a(z) \, b(z^{-1})^{e} \, c(z^{-1})}{z} \right] & = \sum_{m = 1}^{\infty} \dfrac {L_m}{m} = \int_0^1 \left[ \sum_{m=1}^{\infty} L_m \, x^{m-1} \right] dx \\ & = \int_0^1 \left[ \dfrac 1x + \dfrac {1}{\log \, (1-x)} \right] dx \\[5pt] & = \gamma. \end{aligned} \end{equation}
\noindent The Corollary follows now from Theorem \ref{main}. \end{proof}

We conclude with stating that Theorem \ref{main} can also be used to show Riordan matrix representations for $\ln 2$ and $\pi^2/6$. Finding matrix representations of other constants, like $\sqrt{2}$, $\pi$, and the Golden Ratio $\phi$, are of interest.

\section{Acknowledgments}

The authors would like to dedicate this work to the memory of David Harold Blackwell (April 24, 1919 -- July 8, 2010).  Both authors gave the recent annual Blackwell Lectures, organized by the National Association of Mathematicians (NAM) as part of the MAA MathFest.  The first author gave his presentation during the summer of 2009, whereas the second gave his during the summer of 2010.

\bibliographystyle{plain}

\end{document}